\documentclass[twoside,a4paper,reqno,11pt]{amsart} 
\usepackage[top=28mm,right=30mm,bottom=28mm,left=30mm]{geometry}

\usepackage{amsfonts, amsmath, amssymb, mathrsfs, bm, latexsym, stmaryrd, array, mathtools, bold-extra}

\usepackage{color}
\usepackage{xcolor}

\usepackage[pagebackref]{hyperref}

\usepackage{bbm}
\usepackage[all,cmtip]{xy}

\usepackage{enumerate}
\usepackage[shortlabels]{enumitem}

\newtheorem{theorem}{Theorem} 
\newtheorem*{conj*}{Conjecture}

\newtheorem{thm}{Theorem}[section] 
 
\newtheorem{lem}[thm]{Lemma}

\theoremstyle{definition}
\newtheorem{rem}[thm]{Remark}

\newtheorem{ex}[thm]{Example}

\newtheorem{problem}[theorem]{Problem}

\newtheorem{conjecture}[theorem]{Conjecture}

\theoremstyle{definition}

\theoremstyle{remark}

\DeclareMathOperator{\Irr}{Irr}

\DeclareMathOperator{\GL}{GL}

\DeclareMathOperator{\kernel}{Ker}

\newcommand{\Lin}{{\mathrm{ Lin}}}

\newcommand{\la}{\lambda}

\DeclareMathOperator{\Sym}{S}

\newcommand{\HH}{\mathcal{H}}

\newcommand{\Centralizer}{\mathbf{C}}
\newcommand{\Center}{\mathbf{Z}}

\def\norm#1#2{{\bf N}_{#1}(#2)}

\numberwithin{equation}{section}

\newcommand{\Alt}{{\mathrm {A}}}

\def\irr#1{{\rm Irr}(#1)}
\def\C{{\Bbb C}}

\begin{document}

\title[Common zeros of characters]{On common zeros of characters of finite groups}

\author[M. L. Lewis]{Mark L. Lewis}
\address{Department of Mathematical Sciences, Kent State University, Kent, OH 44266, USA}
\email{lewis@math.kent.edu}

\author[L. Morotti]{Lucia Morotti}
\address{Department of Mathematics, University of York, York, YO10 5DD, UK}
\email{lucia.morotti@uni-duesseldorf.de}

\author [E. Pacifici]{Emanuele Pacifici}
\address{Dipartimento di Matematica e Informatica U. Dini, Universit\`a degli Studi di Firenze, Viale Morgagni 67/A, 50134 Firenze, Italy.}
\email{emanuele.pacifici@unifi.it}

\author [L. Sanus]{Lucia Sanus}
\address{Departament de Matem\`atiques, 
Universitat de Val\`encia,
46100 Burjassot, Val\`encia, Spain. }
\email{lucia.sanus@uv.es}

\author[H. P. Tong-Viet]{Hung P. Tong-Viet}
\address{Department of Mathematics and Statistics, Binghamton University, Binghamton, NY 13902-6000, USA}
\email{htongvie@binghamton.edu}

\renewcommand{\shortauthors}{M. L. Lewis et al.}
\begin{abstract}  Let $G$ be a finite group, and let $\irr G$ denote the set of the irreducible complex characters of $G$. An element $g\in G$ is called a \emph{vanishing element} of $G$ if there exists $\chi\in\irr G$ such that $\chi(g)=0$ (i.e., $g$ is a \emph{zero} of $\chi$) and, in this case, the conjugacy class $g^G$ of $g$ in $G$ is called a \emph{vanishing conjugacy class}. In this paper we consider several problems concerning vanishing elements and vanishing conjugacy classes; in particular, we consider the problem of determining the least number of conjugacy classes of a finite group $G$ such that every non-linear $\chi\in\irr G$ vanishes on one of them. We also consider the related problem of determining the minimum number of non-linear irreducible characters of a group such that two of them have a common zero.
\end{abstract}

\thanks{While starting to work on this paper, the second author was working at the Mathematisches Institut of the Heinrich-Heine-Universit\"{a}t D\"usseldorf. Some of the work of the paper was completed during a visit by the first author to Universit\`a degli Studi di Firenze in April 2024 and by the fifth author to Kent State University in October 2023.  Those authors thank the respective universities for their hospitality. \\ The second author was supported by the Royal Society grant URF$\backslash$R$\backslash$221047. The third author is partially supported by InDaM-GNSAGA and by the italian project PRIN 2022-2022PSTWLB - Group Theory and Applications - CUP B53D23009410006. The research of the fourth author is supported by Ministerio de Ciencia e Innovación (PID2022-137612NB-I00 funded by MCIN/AEI/10.13039/501100011033 and `ERDF A way of making Europe') and Generalitat Valenciana CIAICO/2021/163. }





\maketitle

\section{Introduction}

In this paper, all groups are finite unless otherwise stated.  We particularly want to study the zeros of the irreducible characters of $G$.  For the fundamentals of character theory, we refer the reader to \cite{Isaacs}. 

Let $G$ be a group, and $\chi$ an irreducible complex character of $G$. An element $g\in G$ is called a \emph{zero} of $\chi$ if $\chi(g)=0$ and, in this case, we say that $\chi$ vanishes at $g$. The element $g$ is then said to be a \emph{vanishing element} of $G$ if some $\chi\in\irr G$ vanishes at $g$ (this never happens if $g$ lies in the center $\Center(G)$ of $G$) and, in this case, we say that its conjugacy class $g^G$ is a \emph{vanishing conjugacy class} of $G$.

The study of zeros of characters has a long history dating back to work of W. Burnside in 1903 and, since then, it has been the subject of a large number of researches in the literature (we refer the reader to the survey paper \cite{DPS}). The classical result of Burnside in this context states that if $\chi$ is a non-linear irreducible character of $G$ (i.e., an irreducible character whose degree is larger than $1$), then there exists $g\in G$ such that $\chi(g)=0$ (see \cite[Theorem~3.15]{Isaacs}).  Also, J. G. Thompson showed that for each non-linear irreducible character $\chi$ of $G$, more than one third of elements of $G$ satisfy that $\chi(g)=0$ or $\chi(g)$ is a root of unity (see \cite[Theorem~2, Chapter~21]{BZ}). 

Using the classification of finite nonabelian simple groups, G. Malle, G. Navarro and J. B. Olsson showed in \cite{MNO} that every non-linear irreducible character of $G$ vanishes at some prime-power order element. Moreover, given a prime number $p$, it was proved by S. Dolfi, P. Spiga, the third and the fourth authors that if no $p$-element of a group $G$ is vanishing, then $G$ has a normal Sylow $p$-subgroup (\cite{DPSS}); this generalizes the celebrated It\^{o}-Michler Theorem (more precisely, it generalizes the substantial part of it, which states that if $p$ does not divide the degree of any character in $\irr G$ then $G$ has a normal Sylow $p$-subgroup) since any character whose degree is coprime to $p$ cannot vanish on any $p$-element (see \cite[Remark~4.1]{DPSS}).

In this paper, we study certain properties of vanishing conjugacy classes of groups. We first consider the problem of finding the least number of conjugacy classes of a group such that every non-linear irreducible character of the group vanishes at one of them. More precisely, given a group $G$ and a positive integer $k$, we say that $G$ satisfies property $\HH_k$ if there exists a set  $\mathcal{C}$ of  conjugacy classes of $G$ such that $|\mathcal{C}|\leq k$ and every non-linear irreducible character of $G$ vanishes on at least one class in $\mathcal{C}$. Clearly, if $G$ satisfies $\HH_k$, then it also satisfies $\HH_{m}$ for any integer $m \geq k$; we also note that the class of groups satisfying property $\HH_k$ is closed under taking factor groups and direct products (see Remark~\ref{Directproduct}). Given that, we propose the following conjecture.

\begin{conjecture}\label{conj:H3} Let $G$ be a group. Then the following conclusions hold. 
\begin{enumerate}
\item[(a)] $G$ satisfies $\HH_3$.
\item[(b)] If $G$ is solvable, then $G$ satisfies $\HH_2$.
\end{enumerate}
\end{conjecture}

It is proven in \cite{Ladisch} that property $\HH_1$ implies solvability; the first author provides, in \cite{Lewis}, an overview of the research on non-abelian groups which satisfy $\HH_1$.  An element $g\in G$ on which all non-linear irreducible characters of $G$ vanish is been called an \emph{anticentral} or \emph{Camina} element in these two papers (respectively). On the other hand, a group satisfying $\HH_2$ is not necessarily solvable, as shown for instance by the symmetric group ${\rm{S}}_5$.

To give evidence in support of the above conjectures, we are able to verify Conjecture~\ref{conj:H3}(a) for all non-abelian simple groups as well as some groups closely related to the alternating groups and sporadic simple groups. The following result was originally formulated as a conjecture by Z. Arad, J. Stavi and M. Herzog in \cite{ASH}.
 
\begin{theorem}\label{th:simple}
Every non-abelian simple group satisfies $\HH_3.$
\end{theorem}

Recall that for a non-abelian simple group $S$, the group $G$ is {\it almost simple} with socle $S$ if it has a unique minimal normal subgroup and this subgroup is isomorphic to $S$.  We define $G$ to be {\it quasi-simple} if $G/\Center (G)$ is a nonabelian simple group $S$.  We say that the group $G$ is an \emph{almost quasi-simple group} (related to $S$) if $G/\Center(G)$ is almost simple with socle $S$.

\begin{theorem}\label{th:almostsemisimple}
Let $G$ be an almost quasi-simple group whose socle is an alternating or a sporadic simple group. Then $G$ satisfies $\HH_3$. Furthermore, if $G$ is a symmetric group, then in fact $G$ satisfies $\HH_2$.
\end{theorem}

Regarding the ``solvable part" of Conjecture~\ref{conj:H3}, we prove a partial result. Recall that, if $G$ is a solvable group, then the set of subgroups of $G$ defined recursively by 
$${\rm{\bf F}}_1(G)={\rm{\bf F}}(G) {\text{ (the Fitting subgroup of }}G), \quad{\rm{\bf F}}_{i}(G)/{\rm{\bf F}}_{i-1}(G)={\rm{\bf F}}(G/{\rm{\bf F}}_{i-1}(G)) {\text{ for }}i\geq 2,$$ 
is called the \emph{Fitting series} of $G$.  It is an ascending chain of characteristic subgroups that reaches $G$ in a finite number of steps: this number is called the \emph{Fitting height} ${\rm h}(G)$ of $G$. Now, setting conventionally ${\rm{\bf F}}_0(G)=1$, define ${\rm r}_i(G)$ as the minimum size of a generating set of ${\rm{\bf F}}_{i}(G)/{\rm{\bf F}}_{i-1}(G)$ and set ${\rm r}(G)$ to be the sum of the ${\rm r}_i(G)$ for all $i\in\{1,\ldots,{\rm h}(G)\}$.

\begin{theorem}\label{th:hn_pgroup}
Let $G \neq 1$ be a solvable group. Then $G$ satisfies $\HH_{{\rm r}(G)}$.
\end{theorem}

Note that if $p$ is a prime, $G$ is a $p$-group, and ${\bf{\Phi}}(G)$ denotes the Frattini subgroup of $G$, then by Burnside's Basis Theorem (\cite[III, 3.15]{Hu}) the number ${\rm r}(G)$ is such that $p^{{\rm r}(G)}=|G/{\bf{\Phi}}(G)|$. Thus, Conjecture~\ref{conj:H3}(b) holds for every $p$-group $G$ such that $|G/{\bf{\Phi}}(G)|=p^2$.

\medskip
Many properties of zeros of characters can be stated easier using graph theoretic language, and, in a recent paper by N. N. Hung, A. Moret\'o and the second author (\cite{HMM}), two simple undirected graphs related to zeros of characters of a group $G$ have been introduced. The \emph{common-zero graph} $\Gamma_v(G)$ has the set of non-linear irreducible characters of $G$ as its vertex set, and two vertices $\chi,\psi$ are joined by an edge if and only if there exists $g\in G$ such that $\chi(g)=0=\psi(g)$; dually, the set of vanishing conjugacy classes of $G$ is taken as the set of vertices of  the other graph, denoted by $\Delta_v(G)$, where two vertices $g^G,h^G$ are adjacent if and only if there exists $\chi\in\irr G$ such that $\chi(g)=0=\chi(h)$. 

It is shown in Lemma~2.1 of \cite{HMM} that $\Gamma_v(G)$ and $\Delta_v(G)$ always have the same number of connected components, and this number is at most $2$ whenever $G$ is a solvable group; the authors of that paper conjecture that, for every group $G$, these two graphs have at most three connected components and they prove this in the case when $G$ is a non-abelian simple group (\cite[Theorem~D]{HMM}). 

Now, the conclusions of Conjecture~\ref{conj:H3} (if confirmed) easily imply that the \emph{independence number} of $\Gamma_v(G)$ (i.e., the maximum size of a set of pairwise non-adjacent vertices in $\Gamma_v(G)$) is at most $3$ for any group $G$, and it is $2$ if $G$ is solvable. We state next this weaker conjecture.

\begin{conjecture}\label{conj:3characters}
Let $G$ be a group. Then the following conclusions hold.
\begin{enumerate}
\item[(a)] For any four non-linear irreducible characters of $G$, two of them have a common zero.
\item[(b)] If $G$ is solvable then, for any three non-linear irreducible characters of $G$, two of them have a common zero.
\end{enumerate}
\end{conjecture}

Note that the conclusions of the statement above in turn imply that $\Gamma_v(G)$ has at most three connected components for any group $G$, and at most two connected components whenever $G$ is solvable. In particular, Theorem~\ref{th:simple} (which implies Conjecture~\ref{conj:3characters}(a) for non-abelian simple groups) is a generalization of \cite[Theorem~D]{HMM}.

As for Conjecture~\ref{conj:3characters}(b), we prove it in a special case. Recall that a group $G$ is abelian-by-metanilpotent if there exist normal subgroups $N\leq K$ of $G$ such that $N$ is abelian, and both factor groups $K/N$ and $G/K$ are nilpotent.

\begin{theorem}\label{th:abelian-by-metanilpotent} 
Let $G$ be an abelian-by-metanilpotent group. Then, for any three non-linear irreducible characters of $G$, two of them have a common zero.
\end{theorem}

In general, for solvable groups, we provide the following bound for the independence number of $\Gamma_v(G)$ in terms of ${\rm h}(G)$. 

\begin{theorem}\label{th:Fittingheight}
Let $G$ be a solvable group, and let $h$ be the Fitting height of $G$. Then, for any $h+1$ non-linear irreducible characters of $G$, two of them have a common zero.
\end{theorem}

In other words, the independence number of $\Gamma_v(G)$ for a solvable group $G$ is at most ${\rm h}(G)$. It is worth noting that the graph $\Delta_v(G)$ has a very different behavior in this respect; in fact, denoting by ${\rm D}_{2^{n+1}}$ the dihedral group of order $2^{n+1}$ (whose Fitting height is $1$), for every $n\geq 2$ the independence number of $\Delta_v({\rm D}_{2^{n+1}})$ is $n-1$ (see Example~\ref{dihedral}).

\section{Non-solvable groups: a proof of Theorem~\ref{th:simple} and Theorem~\ref{th:almostsemisimple} }


In this section, we consider the results regarding nonsolvable groups.  In particular, we prove Theorem~\ref{th:simple} and Theorem~\ref{th:almostsemisimple}, in reverse order.

\begin{proof}[{Proof of Theorem $\ref{th:almostsemisimple}$}]
We start by proving the first claim: if $G$ is an almost quasi-simple group related to an alternating or sporadic simple group, then $G$ satisfies $\HH_3$. 

This can be checked using GAP \cite{GAP} if $G$ is an almost quasi-simple group related to a sporadic group or $\Alt_n$ with $n\leq 7$. As any irreducible character of a symmetric or alternating group is also an irreducible character of the corresponding double cover, we may from now on assume that $G=2.\Sym_n$ or $2.\Alt_n$ is a double cover of a symmetric or alternating group with $n\geq 8$. Note that, in view of \cite[p. 93]{stem}, when considering zeros of characters it does not matter which double cover of a symmetric group is being considered.
	
For any partition $\mu$ of $n$, let $g_\mu\in 2.\Sym_n$ be the lift of an element of $\Sym_n$ with cycle partition $\mu$. If $\mu$ has an even number of even parts then we also have that $g_\mu\in 2.\Alt_n$.
	
If $n\geq 8$ is even we will show that any non-linear irreducible character of $G$ vanishes on at least one of $g_{(n-1,1)}$, $g_{(n-3,3)}$ or $g_{(n-4,2,1^2)}$. Similarly for any odd $n\geq 9$, taking $g_{(n)}$, $g_{(n-4,2^2)}$ and $g_{(n-5,4,1)}$ instead. 
	
Assume first that $\chi$ is a spin character and let $\la$ be the labeling partition. Then the parts of $\la$ are distinct by \cite[Theorem 7.1 and Corollary 7.5]{stem}. Let $\mu=(n-4,2,1^2)$ or $(n-4,2^2)$ (depending on the parity of $n$). As $\mu$ has repeated parts, $\mu\not=\la$. Further $\mu$ has even parts. It follows by \cite[Theorem 7.1 and Corollary 7.5]{stem} that $\chi(g_\mu)=0$ (formulas for the appearing functions can be found in \cite[p. 118 and Theorem 7.4]{stem}).

So we may now assume that $\chi$ is an irreducible character of $\Alt_n$ or $\Sym_n$. By the second paragraph in the proof of \cite[Proof of Lemma 3.1]{MTV} it is enough to prove that any non-linear irreducible character of $S_n$ vanishes on at least one of $g_{(n-1,1)}$, $g_{(n-3,3)}$ or $g_{(n-4,2,1^2)}$ if $n$ is even, or one of $g_{(n)}$, $g_{(n-4,2^2)}$, $g_{(n-5,4,1)}$ if $n$ is odd.
	
So, consider an irreducible character of $\Sym_n$. Such a character is of the form $\chi^\la$ for a partition $\la$ of $n$ (see for example \cite[Theorem 2.3.15]{JK}). Note that $\chi^\la$ is linear if and only if $\la=(n)$ or $(1^n)$ (see for example \cite[Theorem 2.4.10]{JK}). In the following we will use the Murnaghan-Nakayama formula (see for example \cite[2.4.7]{JK}) without further reference.
	
If $\la\not\in\{(n),(1^n)\}$ and $\chi^\la$  vanishes on neither $g_{(n-1,1)}$ nor $g_{(n-3,3)}$, then $\la\in\{(n-3,2,1),(3,2,1^{n-5})\}$ (as $\la$ must have both an $(n-1)$- and an $(n-3)$-hook). In these cases, it can be checked that $\chi^\la$ vanishes on $g_{(n-4,2,1^2)}$. Similarly, if $\la\not\in\{(n),(1^n)\}$ and $\chi^\la$ vanishes on neither $g_{(n)}$ nor $g_{(n-4,2^2)}$, then
\[\la\in\{(n-1,1),(n-2,1^2),(n-3,1^3),(4,1^{n-4}),(3,1^{n-3}),(2,1^{n-2})\},\]
in which cases $\chi^\la$ always vanishes on $g_{(n-5,4,1)}$.
	
This concludes the proof of the first claim. Since (for any $n\geq 1$) the only partitions with both an $n$- and an $(n-1)$-hook are $(n)$ and $(1^n)$, the second claim that any symmetric group satisfies $\HH_2$ similarly follows (taking conjugacy classes corresponding to $n$- and $(n-1)$-cycles).
\end{proof}

Next, the proof of Theorem~\ref{th:simple}: every non-abelian simple group satisfies $\HH_3$.

\begin{proof}[{Proof of Theorem $\ref{th:simple}$}] 
Using the classification of finite nonabelian simple groups, we consider the following cases.
	
(1) $G$ is a sporadic simple group or $G=\Alt_n$ with $n\ge 5.$ In these cases the theorem holds by Theorem \ref{th:almostsemisimple}.
	
(2) $G$ is a simple group of Lie type. Recall that a pair of conjugacy classes $(C_1,C_2)$ of $G$ is called strongly orthogonal if there exist only two irreducible characters of $G$ that do not vanish on neither $C_1$ nor $C_2.$  Fix a prime power $q=p^f$, where $p$ is a prime and $f\ge 1$ is an integer. Let $n\ge 2$ be an integer. We say that $\ell$ is a primitive prime divisor (or Zsigmondy prime divisor) of $q^n-1$, denoted by $\ell=\ell(n)$, if $\ell$ divides $q^n-1$ but $\ell$ does not divide $q^k-1$ for all positive integers $k<n.$ By Zsigmondy's Theorem \cite{Zsig}, $\ell(n)$ always exists for any pair $(n,q)$ with $n\ge 2$, except for the cases $(n,q)=(6,2)$ or $n=2$ and $q+1$ a power of $2$. 
	
It follows from the proof of Theorems $2.1-2.6$ in \cite{MSW} for simple classical groups and Theorem 10.1 in \cite{LM} for exceptional groups of Lie type that $G$ always possesses a pair of strongly orthogonal conjugacy classes $(C_1,C_2)$, except possibly for ${\rm O}_{2n}^+(q),$ where $q$ is a power of some prime $p$ and $n\ge 4$ is an even integer. Thus if $G$ is not an exception, there is a unique non--trivial character which does not vanish on neither $C_1$ nor $C_2$. Since this character is non-linear it has at least one vanishing conjugacy class $C_3$, so we are done.

We now deal with the  remaining case.
	
Let $q=p^f$ be a power of a prime $p$ and let $\tilde{G}=\tilde{G}(q)$ be the group of fixed points of a simple algebraic group of simply connected type $D_{n}$ with $n\ge 4$ an even integer under a Frobenius map $F$. So $G=\tilde{G}/\Center(\tilde{G})\cong {\rm O}_{2n}^+(q)$. Notice that if $\tilde{G}$ satisfies the conclusion of Theorem \ref{th:simple} then so does $G$.
	
By \cite[3A]{Malle}, $\tilde{G}$ has two maximal tori $T_1$ and $T_2$ of order $(q^{n-1}+1)(q+1)$ and $(q^{n-1}-1)(q-1)$, respectively. If $G\cong {\rm O}_8^+(2)$, then we can take $C_1,C_2$ and $C_3$ to be conjugacy classes of elements labeled by $7A,8A$ and $9A$, respectively by using GAP \cite{GAP}. So, we can assume that $(n,q)\neq (4,2)$ and recall that $n\ge 4$ is even. It follows that the primitive prime divisors $\ell_1=\ell(2n-2)$ and $\ell_2=\ell(n-1)$ exist. Let $C_1$ and $C_2$ be the conjugacy classes of elements of order $\ell_1$ and $\ell_2$, respectively. By \cite[3G]{Malle}, if $\chi$ is a non--trivial irreducible character of $G$ that does not vanish on both $C_1$ and $C_2$, then $\chi$ is a Steinberg character of degree $|G|_p$ or $\chi$ is one of the two unipotent characters $\psi_1$ and $\psi_2$, labeled by the symbols
	
\[\begin{pmatrix} n-1\\1  \end{pmatrix} \text{ and } \begin{pmatrix} 0&\dots&n-3&n-1\\1&\dots&n-2&n-1  \end{pmatrix}.\]
Let $\ell_3=\ell(2n-4)$. Since $n\ge 4$ is even, $\ell_3$ always exists. Using \cite[13.8]{Carter} or \cite[3G]{Malle}, we can see that both $\psi_1$ and $\psi_2$ have $\ell_3$-defect zero.
Note that if $\chi$ is an irreducible character of $G$ having $r$-defect zero for some prime $r$, that is, $|G|/\chi(1)$ is coprime to $r$, then $\chi$ vanishes on every $r$-singular element of $G$, i.e., elements whose orders are divisible by $r$. To complete the proof, we need to show that $G$ contains an element  whose  order is divisible by $p\ell_3$ and we can take $C_3$ to be the conjugacy class of $G$ containing such elements.
	
It follows from  \cite[Tables 8.50, 8.66]{BHRD}  for $n\in \{4,6\}$ and \cite[Proposition 4.1.6]{KL} for $n\ge 8$ that $G$ contains a subgroup of the form $\Omega_4^-(q)\times \Omega_{2n-4}^-(q)$. Now the first factor $\Omega_4^-(q)$ 
contains an element of order $p$ and $\Omega_{2n-4}^-(q)$ contains an element of order $\ell_3$ since $q^{n-2}+1$ dividing $|\Omega_{2n-4}^-(q)|$. Thus $\Omega_4^-(q)\times \Omega_{2n-4}^-(q)$ and hence $G$ possesses an element of order $p\ell_3$. The proof is now complete.
\end{proof}

\section{Solvable groups: proofs of Theorems~\ref{th:hn_pgroup}, ~\ref{th:abelian-by-metanilpotent}, and ~\ref{th:Fittingheight}}

In this section, we consider several results regarding solvable groups.  The following lemma will be a key tool for the results of this section.

\begin{lem}\label{alex1}
Let $G$ be a group, and let $K$ be a proper normal subgroup of $G$ such that $G/K$ is nilpotent. Also, let $\chi$ be an irreducible character of $G$ such that $\chi_K$ is not irreducible. Then there exists a proper normal subgroup $L$ of $G$ such that $L$ contains $K$ and, for every $g\in G\setminus L$, we have $\chi(g)=0$.
\end{lem}

\begin{proof}
Consider first the case when $\chi_K$ is not homogeneous. Then, by Clifford's Theory, there exist a proper subgroup $I$ of $G$, with $I \supseteq K$, and $\theta \in \irr I$ such that $\chi = \theta^G$. If $L/K$ is a maximal subgroup of $G/K$ containing $I/K$ then, as $G/K$ is nilpotent, $L/K$ is a (proper) normal subgroup of $G/K$. Hence $L$ is a proper normal subgroup of $G$, and $\chi=(\theta^L)^G$ takes the value $0$ on every element of $G\setminus L$, as wanted.
	
It remains to consider the case when $\chi_K=e\theta$, for a suitable character $\theta \in \irr K$ and a suitable integer $e > 1$. Let $R$ be a representation of $G$ affording $\chi$ and such that $R_K = I_e \otimes T$, where $I_e$ denotes the trivial representation of degree $e$ and $T$ is a representation affording $\theta$. By Theorem~21.2 of \cite{Hu}, there exist irreducible projective representations $P$ and $Q$ of $G$ such that, for every $g\in G$, we have $R(g) = P(g) \otimes Q(g)$ and $P_K = I_e$. Now, the composition map given by $P$ followed by the natural homomorphism of ${\rm{GL}}_e (\C)$ onto ${\rm{PGL}}_e(\C)$ is a homomorphism whose kernel contains $K$, therefore it is well defined on $G/K$; in other words, $P$ can be regarded as a projective representation of $G/K$ (up to multiplying it by a suitable map from $G$ to $\C^{\times}$). 

Also, denoting by $M$ the Schur multiplier of $G/K$ and by $\Gamma$ a Schur representation group for $G/K$, we can find an ordinary irreducible representation $X$ of $\Gamma$ such that, for every $\gamma \in \Gamma$, we have $X(\gamma) = P(\gamma M) \mu(\gamma)$ for a suitable map $\mu : \Gamma \rightarrow \C^{\times}$. Denote by $\xi \in \irr{\Gamma}$ the (non-linear) character afforded by $X$: the nilpotency of $\Gamma$ guarantees that $\xi$ is induced by a linear character of $\Gamma$, hence it is induced from a proper subgroup of $\Gamma$ and, as in the paragraph above, it is in fact induced from a proper normal subgroup $\Lambda$ of $\Gamma$; in particular, $\xi$  takes the value $0$ on every element of $\Gamma\setminus\Lambda$ and, as an immediate consequence, we see that $\Lambda$ contains the central subgroup $M$. Finally, let $L$ be the (proper, normal) subgroup of $G$ such that $L/K=\Lambda/M$: if $g$ lies in $G\setminus L$ and $\gamma\in\Gamma$ is such that $gK=\gamma M$ (hence $\gamma\not\in\Lambda$), then we get \[0=\xi(\gamma)={\rm{tr}}(X(\gamma))={\rm{tr}}(P(g))\mu(\gamma).\] It follows that ${\rm{tr}}(P(g))=0$, thus $\chi(g)={\rm{tr}}(P(g))\cdot{\rm{tr}}(Q(g))=0$. The proof is complete.  
\end{proof}

We are now in a position to prove Theorem~\ref{th:hn_pgroup}: every non-trivial solvable group $G$ satisfies $\HH_{{\rm{r}}(G)}$. For the definition of ${\rm{r}}(G)$, we refer the reader to the paragraph of the Introduction preceding the statement of Theorem~\ref{th:hn_pgroup}.

\begin{proof}[{Proof of Theorem $\ref{th:hn_pgroup}$}] Since every abelian group (vacuously) satisfies $\HH_1$ just by choosing the conjugacy class of the identity element, we can assume that $G$ is non-abelian. Denoting by $h$ the Fitting height of $G$, let $K=\mathbf{F}_{h-1}(G)$ be the penultimate term of the Fitting series of $G$; thus, $G/K$ is a nilpotent group. Also, set $r_h={\rm r}_h(G)$ and, adopting the bar convention for the factor group $G/K$, let $\{g_1,\ldots,g_{r_h}\}$ be a subset of $G$ such that $\{\overline{g_1},\ldots,\overline{g_{r_h}}\}$ is a generating set of $G/K$.

Consider an irreducible character $\chi$ of $G$, and assume that $\chi_K$ is not irreducible; then, by Lemma~\ref{alex1}, there exists a proper normal subgroup $L$ of $G$, containing $K$, such that $\chi$ vanishes off $L$. As $\overline{L}$ is a proper subgroup of $\overline{G}$, clearly there exists $i\in\{1,\ldots,r_h\}$ such that $\overline{L}$ does not contain $\overline{g_i}$, hence $g_i$ lies in $G\setminus L$ and $\chi(g_i)=0$. On the other hand, assume now that $\chi\in\irr G$ is non-linear and restricts irreducibly to $K$ (note that this implies $K\neq 1$); arguing by induction on the order of the group and observing that $\{\mathbf{F}_{1}(G),\ldots,\mathbf{F}_{h-1}(G)\}$ is the Fitting series of $K$, we have that there exists a set $Y$ of elements of $K$ having size at most ${\rm r}(K)={\rm r}(G)-r_h$ such that $\chi_K$ (hence $\chi$) vanishes on at least an element of $Y$. Considering then the subset $X=Y\cup\{g_1,\ldots,g_{r_h}\}$ of $G$, we see that $|X|\leq{\rm{r}}(G)$ and every non-linear irreducible character of $G$ vanishes on at least an element of $X$, yielding the desired conclusion.
\end{proof}

Next, we derive another immediate consequence of Lemma~\ref{alex1}.

\begin{lem}\label{alex2} 
Let $G$ be a group, and let $K$ be a proper normal subgroup of $G$ such that $G/K$ is nilpotent. If $\chi_1$ and $\chi_2$ are characters in $\irr G$ whose restrictions to $K$ are not irreducible, then there exists $g\in G$ such that $\chi_1(g)=0=\chi_2(g)$. 
\end{lem}

\begin{proof} 
By Lemma~\ref{alex1}, we can find proper normal subgroups $L_1$, $L_2$ of $G$ such that $\chi_1$ vanishes on every element of $G\setminus L_1$ and $\chi_2$ vanishes on every element of $G\setminus L_2$. Taking into account that $L_1\cup L_2$ is properly contained in $G$, an element $g$ which satisfies the desired conclusion can be chosen as any element of $G\setminus (L_1\cup L_2)$.
\end{proof}



We are now ready to prove Theorem~\ref{th:abelian-by-metanilpotent}. If $\chi_1$, $\chi_2$ and $\chi_3$ are non-linear irreducible characters of an abelian-by-metanilpotent group $G$, then two of them have a common zero (in other words, the independence number of the graph $\Gamma_v(G)$ is at most $2$).

\begin{proof}[{Proof of Theorem~$\ref{th:abelian-by-metanilpotent}$}] 
Let $K$, $N$ be normal subgroups of $G$ such that $N\subseteq K$, $N$ is abelian, and the factor groups $K/N$, $G/K$ are nilpotent. Also, let $X$ be a set of non-linear irreducible characters of $G$ such that $|X|=3$. Since $N$ is abelian, every character in $X$ does not restrict irreducibly to $N$ because it is non-linear. On the other hand, if two among these characters do not restrict irreducibly to $K$, then the conclusion follows from Lemma~\ref{alex2}. We conclude that at least two characters in $X$ restrict irreducibly to $K$, but not to $N$, and another application of Lemma \ref{alex2} easily yields the conclusion.
\end{proof}

Finally, we prove Theorem~\ref{th:Fittingheight}. Let $G$ be a solvable group and let $h$ be the Fitting height of $G$. Then, for any $h+1$ non-linear irreducible characters of $G$, two of them have a common zero. Again, this can be rephrased by saying that the independence number of $\Gamma_v(G)$ is at most $h$.   As mentioned in the Introduction, a similar statement does not hold for the graph $\Delta_v(G)$, since there are solvable groups $G$ with a prescribed value for the Fitting height but whose graph $\Delta_v(G)$ has an arbitrarily large independence number; this will be discussed in Example~\ref{dihedral}.


\begin{proof}[{Proof of Theorem~$\ref{th:Fittingheight}$}] 
Considering a set $X$ of non-linear irreducible characters of $G$ such that $|X|=h+1$, we argue by induction on $h$. If $h = 1$, then Lemma~\ref{alex2} applied with $K=1$ (or Theorem~4.1 of \cite{HMM}) yields the conclusion. So, let us assume $h \geq 2$, and that the claim holds for groups having Fitting height $h - 1$. Denoting by $K$ the penultimate term of the Fitting series of $G$, by Lemma~\ref{alex2} the desired conclusion follows if there are two characters in $X$ that do not restrict irreducibly to $K$; therefore we can assume that there exists a subset $X_0$ of $X$, with $|X_0|=h$, such that all restrictions to $K$ of the characters in $X_0$ are irreducible. Since the Fitting height of $K$ is $h-1$, by the inductive hypothesis there exist $\chi_1$, $\chi_2$ in $X_0$ and an element $g\in K$ such that $\chi_1(g)=0=\chi_2(g)$, which concludes the proof.
\end{proof}

\begin{ex}\label{dihedral}
Let $G_{n}$ denote the dihedral group ${\rm{D}}_{2^{n+1}}$ of order $2^{n+1}$, for $n\geq 2$: we claim  that (the Fitting height of every group $G_n$ is $1$ and) the independence number of $\Delta_v(G_n)$ is $n-1$.

In fact, let $x$ be a generator of the group of rotations $R\subseteq G_n$ (thus, $x$ is an element of order $2^n$) and let $\epsilon$ denote a primitive $2^n$-th root of unity. The irreducible characters of $R$ can be parametrized as the group homomorphisms $\theta_i:R\rightarrow\C^\times$ defined by $\theta_i(x)=\epsilon^i$, for $i\in\{0,\ldots,2^n-1\}$: among these, $\theta_0$ and $\theta_{2^{n-1}}$ are $G_n$-invariant, whereas all the others induce irreducibly to $G_n$ (note that, for every $y\in G_n\setminus R$ and for every $i\in\{0,\ldots,2^n-1\}$, we have $\theta_i^y=\overline{\theta_i}=\theta_{2^n-i}$). It is then easy to see that the non-linear irreducible characters of $G_n$ are the characters in the set $\{\chi_i=\theta_i^G\mid i\in\{1,\ldots,2^{n-1}-1\}\}$.

Now, the character $\chi_i$ vanishes at $x^j\in R$ if and only if $\theta_i(x^j)+\overline{\theta_i(x^j)}=\epsilon^{ij}+\overline{\epsilon^{ij}}=0$, and this holds precisely when $ij\equiv 2^{n-2}$ (mod $2^{n-1}$). Defining $d$ to be the $2$-part of $i$ (note that $d$ is at most $2^{n-2}$) and $w\in\{0,\ldots,2^n-1\}$ to be the unique (odd) number such that $w\cdot (i/d)\equiv 1$ (mod $2^{n-1}/d$), the solutions in the variable $j\in\mathbb{Z}$ of the congruence $$ij\equiv 2^{n-2} \text{ (mod }2^{n-1})$$ are precisely the numbers of the form $$j=\frac{2^{n-2}}{d}\cdot(2k+w)\quad {\text{ for }}k\in\mathbb{Z},$$ hence the $2$-part of $j$ is $2^{n-2}/d$ and, consequently, the order of $x^j$ is $4d$. In particular, the order of any $x^j\in R$ such that $\chi_i(x^j)=0$ is a $2$-power larger than $2$ which is uniquely determined by $i$. Our conclusion so far is that, for every non-linear irreducible character $\chi$ of $G_{n}$, the zeros of $\chi$ lying inside $R$ have all the same order. Finally, recalling that every non-central element of a nilpotent group is a vanishing element of the group (see \cite[Theorem~B]{INW}) and setting $x_t$ to be an element of $R$ having order $2^t$ for $t\in\{2,\ldots, n\}$, the conjugacy classes $x_t^{G_n}$ are (pairwise distinct) vanishing classes of $G_n$ such that no irreducible character of $G_n$ vanishes on two of them. In other words, we constructed an independent set of $\Delta_v(G_n)$ of size $n-1$.
\end{ex}

\section{Final remarks}

We conclude this paper with some remarks, beginning with the following observation already mentioned in the Introduction.

\begin{rem}\label{Directproduct}
Given a positive integer $k$ and a group $G\in {\mathcal H}_k$, it is straightforward to see that any factor group of $G$ satisfies ${\mathcal H}_k$ as well. It is also not difficult to check that the class of groups satisfying property ${\mathcal H}_k$ is closed under taking direct products.

In fact, let $X$ and $Y$ be groups in ${\mathcal H}_k$;  by definition, there exist subsets $\{x_1,\ldots,x_k\}$ of $X$ and $\{y_1,\ldots,y_k\}$ of $Y$ such that every non-linear irreducible character of $X$ (of $Y$, respectively) vanishes on at least one of the $x_i$ (of the $y_i$, respectively). Now, consider the subset $
\{g_1=(x_1,y_1),\ldots,g_k=(x_k,y_k)\}$ of the direct product $X\times Y$, and let $\chi$ be a non-linear irreducible character of $X\times Y$; hence there exist $\mu\in\irr X$ and $\nu\in\irr Y$ such that $\chi=\mu\times \nu$, and at least one among $\mu$, $\nu$ is non-linear. Assuming that $\mu(1)>1$ and that $j\in\{1,\ldots, k\}$ satisfies $\mu(x_j)=0$, we have $\chi(g_j)=\mu(x_j)\nu(y_j)=0$, which proves the claim.
\end{rem}

In Theorem~2 we have already seen that every symmetric group satisfies $\mathcal{H}_2$. It can be checked that other examples of non-solvable groups lying in $\mathcal{H}_2$ are
${\rm Suz}.2$, $3.{\rm Suz}.2$, $\Alt_6.2_3$, $3.\Alt_6.2_3$, $\Alt_6.2^2$, $3.\Alt_6.2^2$, and in all these groups we can actually find two elements \emph{of prime-power order} as representatives for the relevant conjugacy classes. 
In the following remark, we have a closer look at the latter situation. 

\begin{rem} 
If the group $G$ has two conjugacy classes $C_1$ and $C_2$ of elements of prime-power orders, say for the distinct primes $p$ and $q$, such that every non-linear irreducible character of $G$ vanishes on either $C_1$ or $C_2$, then by  \cite[Remark 4.1]{DPSS} the degree of every non-linear irreducible character of $G$ is divisible by either $p$ or $q$. This can be regarded as a variation of the property that every non-linear irreducible character of a group $G$ has a degree divisible by a given prime $p$ (under this hypothesis, a celebrated theorem by J. Thompson proves that $G$ has a normal $p$-complement).

Now, setting $\pi=\{p, q\}$, assuming that $G$ is a $\pi$-separable group and denoting by $H$ a Hall $\pi$-subgroup of $G$, a result by G. Navarro, N. Rizo, and the fourth author shows that every non-linear irreducible character of $G$ has a degree divisible by either $p$ or $q$ if and only if $G = G' \norm G H$ and $G' \cap \norm G H = H'$ (see  \cite[Theorem~A]{NRS} applied to the complement $\pi'$ of $\pi=\{p,q\}$ and to $\emptyset$). 
Also, in \cite[Theorem~3.6]{GSFV} the authors prove that for $G=\GL_n^\epsilon(r^a)$, $\pi=\{p,q\}$ as above and $\epsilon\in\{\pm 1\}$, every non-linear irreducible character of $G$ has a degree divisible by either $p$ or $q$ if and only if there is some $k\ge 0$ such that $(r,n)=(p,q^k)$ and $q\mid (p^a-\epsilon)$, up to reordering $p$ and $q$. 

However, we note that if every non-linear irreducible character of a group $G$ has a degree divisible by either $p$ or $q$ (for $p$ and $q$ suitable distinct primes), then it is not true in general that there exist two conjugacy classes $C_1$, $C_2$ of elements of $p$-power order and, respectively, $q$-power order, such that every non-linear irreducible character of $G$ vanishes on either $C_1$ or $C_2$;  as an example, we can consider $G = \GL_2(5)$ with $p=2$ and $q=5$.

\end{rem}

We close with another graph interpretation of zeros of characters.

\begin{rem}
As mentioned in the Introduction, some properties of zeros of characters can be conveniently rephrased in a graph theoretic language, and, given a group $G$, we already discussed some features of the graphs $\Gamma_v(G)$ and $\Delta_v(G)$ in this respect. 

Another graph, which ``summarizes" $\Gamma_v(G)$ and $\Delta_v(G)$, can be be defined as follows: consider the (simple, undirected) bipartite graph $\Theta(G)$ whose vertices are the non-linear irreducible characters and the non-central conjugacy classes of $G$; two conjugacy classes as well as two characters are never adjacent in $\Theta(G)$, whereas a character $\chi$ and a conjugacy class $g^G$ are adjacent if and only if $\chi(g)=0$. 

Note that, by Burnside's Theorem, no character can be an isolated vertex of this graph. On the other hand there can be conjugacy classes that are isolated vertices of $\Theta(G)$ (as one can already see in the symmetric group ${\rm S}_3$, considering the conjugacy class of an element of order $3$): these are the non-central conjugacy classes of the \emph{non-vanishing elements}, originally introduced  in \cite{INW} and then studied by many researchers. We also note that, removing the isolated vertices from the vertex set of $\Theta(G)$, we obtain a bipartite graph whose vertex set is precisely the union of the vertex sets of $\Gamma_v(G)$ and $\Delta_v(G)$ and that is of course very much related to these graphs.

Recall that an element $g\in G$ is called a Camina element (or anticentral element) (\cite[Lemma~2.1]{Lewis} or \cite[Proposition~1.1]{Ladisch}) if $|\Centralizer_G(g)|=|G:G'|$ or, equivalently, if every non-linear irreducible character of $G$ vanishes at $g$. In view of the above discussion, this can be expressed by saying that the vertex $g^G$ is adjacent in $\Theta(G)$ to all the vertices of the other part in the bipartition, i.e., to all the non-linear irreducible characters of $G$; in other words, $g^G$ is a vertex having the maximum possible degree in $\Theta(G)$. It is known that groups with such an element are solvable (see \cite[Theorem~4.3]{Ladisch}). 

The dual case for Camina elements is the case where a character vanishes on all noncentral elements.  Such a character is known to be fully-ramified with respect to the center, and a group having such a character is called a central type group.  DeMeyer and Janusz have proved that the Sylow subgroups of central type groups have central type \cite{DeJa} and building on the work of Gagola \cite{gag1}, Howlett and Isaacs prove that central type groups are solvable \cite{HI}.

In a paper in preparation \cite{inprep}, we study groups that are one step away from the groups in the previous two paragraphs.  In particular, we consider groups $G$ with an element $g \in G$ where all but one nonlinear irreducible character $G$ vanishes on $g$.  We also consider a group $G$ with an irreducible character that vanishes on all but one noncentral conjugacy class of $G$.  We will see that these results have a common generalization.
\end{rem}


\begin{thebibliography}{100}
\bibitem{ASH} Z. Arad, J. Stavi\ and\ M. Herzog, Powers and products of conjugacy classes in groups, in {\it Products of conjugacy classes in groups}, 6--51, Lecture Notes in Math., 1112, Springer, Berlin, 1985.

\bibitem{BZ} Y.G. Berkovich, \'E.M. \v{Z}mud' {\it Characters of finite groups. Part 2}, translated from the Russian manuscript by P. Shumyatsky [P. V. Shumyatsky], V. Zobina and Y.G. Berkovich, Translations of Mathematical Monographs, 181, Amer. Math. Soc., Providence, RI, 1999. 



\bibitem{BHRD} J.N. Bray, D.F. Holt\ and\ C.M. Roney-Dougal, {\it The maximal subgroups of the low-dimensional finite classical groups}, London Mathematical Society Lecture Note Series, 407, Cambridge Univ. Press, Cambridge, 2013.


\bibitem{Carter} R.W. Carter, {\it Finite groups of Lie type}, Pure and Applied Mathematics (New York), John Wiley \& Sons, Inc., New York, 1985.


\bibitem{DeJa} F. R. DeMeyer, G. J. Janusz, Finite groups with an irreducible representation of large degree, \emph{Math. Z.} {\bf 108} (1969), 145-153.

\bibitem{DPS} S. Dolfi, E. Pacifici, L. Sanus, \emph{On zeros of characters of finite groups}, Group theory and computation, Indian Stat. Inst. Ser., Springer, Singapore (2018), 41--58.

\bibitem{DPSS} S. Dolfi, E. Pacifici, L. Sanus, P. Spiga, On the orders of zeros of irreducible characters. \emph{J. Algebra} \textbf{321} (2009), no.~1, 345--352.




\bibitem{GAP} The GAP Group, GAP  Groups, Algorithms, and Programming, {\tt http://www.gap-system.org.}


\bibitem{gag1} S. M. Gagola, Characters fully ramified over a normal subgroup, \emph{Pacific J. Math.} {\bf 55} (1974), 107-126.

\bibitem{GSFV}
 E. Giannelli, A.~A. Schaeffer~Fry\ and\ C. Vallejo, Characters of $\pi'$-degree, \emph{Proc. Amer. Math. Soc.} {\bf 147} (2019), no.~11, 4697--4712.
 
\bibitem{HI} R. B. Howlett\ and\ I. M. Isaacs, On groups of central type, \emph{Math. Z.} \textbf{ 179} (1982), no.~4, 555--569.

\bibitem{HMM} N. N. Hung, A. Moret{\'o}, L. Morotti, Common zeros of irreducible characters, Journal of the Australian Mathematical Society. Published online 2023:1-25.  doi:10.1017/S1446788723000216.

\bibitem{Hu} B. Huppert, \emph{Character theory of finite groups}, De Gruyter, Berlin, 1998.

\bibitem{Isaacs} I.M. Isaacs, \emph{Character theory of finite groups}. AMS Chelsea Publishing, Providence, RI, 2006.  


\bibitem{INW} I.M. Isaacs, G. Navarro, T.R. Wolf, Finite group elements where no irreducible character vanishes, \emph{J. Algebra} \textbf{ 222} (1999), no.~2, 413--423.

 
 \bibitem{JK} G.D. James, A. Kerber, {\it The representation theory of the symmetric group}, Encyclopedia of Mathematics and its Applications, 16, Addison-Wesley Publishing Co., Reading, MA, 1981.

\bibitem{KL} P.B. Kleidman, M.W. Liebeck, {\it The subgroup structure of the finite classical groups}, London Mathematical Society Lecture Note Series, 129, Cambridge Univ. Press, Cambridge, 1990.

\bibitem{Ladisch} F. Ladisch, Groups with anticentral elements, Comm. Algebra {\bf 36} (2008), no.~8, 2883--2894. 

\bibitem{LM} F. L\"{u}beck, G. Malle, $(2,3)$-generation of exceptional groups, J. London Math. Soc. (2) {\bf 59} (1999), no.~1, 109--122.

\bibitem{Lewis} M. L. Lewis, Camina groups, Camina pairs, and generalizations, in \emph{ Group theory and computation}, 141--173, Indian Stat. Inst. Ser, Springer, Singapore.

\bibitem {inprep} M. L Lewis, L. Morotti, E. Pacifici, L. Sanus, H. P. Tong-Viet,  in preparation.



\bibitem{Malle} G. Malle, Almost irreducible tensor squares, Comm. Algebra {\bf 27} (1999), no.~3, 1033--1051.

\bibitem{MNO} G. Malle, G. Navarro, J.B. Olsson, Zeros of characters of finite groups, J. Group Theory {\bf 3} (2000), no.~4, 353--368. 

\bibitem{MSW} G. Malle, J. Saxl, T.S. Weigel, Generation of classical groups, Geom. Dedicata {\bf 49} (1994), no.~1, 85--116.

\bibitem{MTV} L. Morotti, H. P. Tong-Viet, Proportions of vanishing elements in finite groups, Israel J. Math. {\bf 246} (2021), no.~1, 441--457. 

\bibitem{NRS} G. Navarro, N. Rizo, L. Sanus, Character degrees in separable groups, \emph{Proc. Amer. Math. Soc.} \textbf{ 150} (2022), 2323--2329.





\bibitem{stem} J. Stembridge, Shifted tableaux and the projective representations of symmetric groups, {\em Adv. Math.} {\bf 74} (1989), 87--134.


\bibitem{Zsig} 
K. Zsigmondy, Zur {T}heorie der {P}otenzreste, Monatsh. Math. Phys.
\textbf{3} (1892), 265--284.

\end{thebibliography}
\end{document}